%% file: liouville_rewriting.tex
\documentclass[11pt]{amsart}
\usepackage[T1]{fontenc}
\usepackage[utf8]{inputenc}
\usepackage{lmodern}
\usepackage[margin=1in]{geometry}
\usepackage{amsmath,amssymb,amsthm,mathtools}
\usepackage{microtype}
\usepackage{tikz}
\usetikzlibrary{arrows.meta}
\definecolor{ink}{RGB}{35,43,55}
\usetikzlibrary{angles,quotes}
\usepackage{enumitem}
\usepackage[colorlinks=false,linkcolor=blue,citecolor=blue,urlcolor=blue]{hyperref}
\mathtoolsset{showonlyrefs}
\usepackage{fancyhdr}
\hypersetup{pdftitle={ A Liouville theorem for geometric triangulations in the plane},pdfsubject={Bounded harmonic functions for the cotangent Laplacian}}
\DeclareMathOperator{\diam}{diam}
\DeclareMathOperator{\area}{area}
\newtheorem{theorem}{Theorem}[section]
\newtheorem{lemma}[theorem]{Lemma}

\newtheorem{corollary}[theorem]{Corollary}
\newtheorem*{recurrence}{Theorem~\ref{tran}}
\theoremstyle{definition}
\newtheorem{definition}[theorem]{Definition}
\theoremstyle{remark}
\newtheorem{remark}[theorem]{Remark}

\newcommand{\R}{\mathbb{R}}

\newcommand{\wind}{\operatorname{wind}}

\title{The Liouville Property for Geometric  Graphs in the Plane
}
\author{Kuijun Liu, Feng Luo, Yanwen Luo, Zhenghao Rao}
\date{}
\address{Department of Mathematics, Rutgers University--New Brunswick, Piscataway, NJ 08854, USA}
\email{kuijun.liu@rutgers.edu}

\address{Department of Mathematics, Rutgers University--New Brunswick, Piscataway, NJ 08854, USA}
\email{fluo@math.rutgers.edu}

\address{School of Mathematics and Physics, University of Science and Technology Beijing, Beijing, 100083, P.R. China} \email{yanwen.luo@ustb.edu.cn}

\address{Department of Mathematics, Rutgers University--New Brunswick, Piscataway, NJ 08854, USA}
\email{zhenghao.rao@rutgers.edu}



\usepackage{lineno}

\begin{document}
\maketitle
\vspace{-2.5em}
\begin{abstract}

We show that every bounded discrete harmonic function with respect to the cotangent Laplacian on a Delaunay triangulation of the Euclidean plane is constant; equivalently, the cotangent-weighted triangulation has the Liouville property. 
Our main theorem establishes this property for a more general class of connected weighted geometric graphs in the plane. 
Furthermore, 
we show that a planar Delaunay triangulation with cotangent weights could be transient but is recurrent if edge lengths are uniformly bounded. 
\end{abstract}

\section{Introduction}

The classical Liouville theorem in complex analysis states that a bounded harmonic function on the plane is a constant. Its discrete version takes the form of a discrete harmonic function on a network. Recall that a network is a pair $(G, c)$ where $G=(V, E)$ is a graph and $c: E \to [0, \infty)$ is a non-negative undirected edge weight with $c_{ij}=c(\{i,j\})$. The discrete Laplacian $\Delta_c$ on the network is a linear operator sending functions defined on $V$ to functions defined on $V$. It is given by the formula

\begin{equation}\label{eq:lap}
\Delta_c(f)(i) =\sum_{j: \{i,j\} \in E} c_{ij}(f(j)-f(i)).
\end{equation}
A function $f: V \to \R$ with $\Delta_c(f)=0$ is called \emph{discrete $c$-harmonic} or simply \emph{discrete harmonic} when the weight $c$ is obvious.
A network is said to have the \emph{Liouville property} if every bounded discrete harmonic function is a constant. The goal of the paper is to establish a fairly general condition on geometric planar graphs that implies the Liouville property.

Recall a triangulation of the plane $\R^2$ is geometric if each triangle is a Euclidean triangle. Note that, by definition, a triangulation of $\R^2$ is always locally finite: each compact set intersects only a finite number of vertices, edges, and triangles. Similarly, a geometric planar graph, i.e., a planar framework,  is an embedded graph in $\R^2$ whose edges are line segments. Therefore, the one skeleton $T^{(1)}$ of $T$ is a natural geometric planar graph. We use $V=V(T)$ and $E$ to denote the set of vertices and edges in $T$, respectively, and use $p \in V$ to represent the vertex or its position in $\R^2$ for a geometric graph.   Our main result is the following

\begin{theorem}\label{thm:main-general}
Let $T$ be a geometric triangulation of $\R^2$ equipped with a nonnegative edge weight $c$, and $G_+=(V, E_+)$ with $E_+=\{ \{p,q\} \in E: c_{pq}>0\}$ be a subgraph of the 1-skeleton of $T$. 
If the restrictions of affine functions to the vertices $V$ of $T$ are discrete harmonic, i.e., for every $p\in V$,
\begin{equation}\label{eq:es}  
\sum_{q: \{p,q\} \in E} c_{pq}(p-q)=0, 
\end{equation}
then 
each component of $G_+$ with edge weights $c$ has the Liouville property.
\end{theorem}

Discrete harmonicity of an affine function appeared in the Maxwell-Cremona correspondence of planar geometric graphs with edge weights.
In this setting, suppose $(G,c)$ is a geometric planar graph with $c=\{c_{pq}\}$ as the edge-stresses. The stress $c$ is in the \emph{equilibrium state} if it satisfies equation \eqref{eq:es} for all $p\in V$,
which is equivalent to the condition that the restrictions of all affine functions to $V$ are discrete harmonic. For a finite geometric planar graph $G$, the Maxwell-Cremona theorem says that $c$ is in the equilibrium state if and only if one can assign to each connected component $\Omega$ of $\R^2 -G$ a point $g_{\Omega} \in \R^2$, called the dual vertex, 
such that the \emph{reciprocal dual equation} 
\begin{equation}\label{eq:dual} g_L -g_R = c_{pq}J (q-p)
\end{equation}
holds for each edge $\{p,q\} \in E$. Here $g_L$ (resp. $g_R$) is the point associated with the region to the left (resp. right) of the edge oriented from $p$ to $q$ and $J(x,y)=(-y,x)$ is the counterclockwise rotation by $\pi/2$.  Note that for any planar network $(G, c)$, even if it has infinitely many vertices,  equation  \eqref{eq:dual} implies the equilibrium stress equation \eqref{eq:es}. Indeed, summing up the dual equation \eqref{eq:dual} over all edges from a vertex $p$ %\ZR{($p$?)}
leads to the equilibrium equation \eqref{eq:es}. See a systematic treatment of Maxwell-Cremona correspondence in Chapter 13 of \cite{richtergebert1996realization}. 

From this new viewpoint,  we can rephrase Theorem \ref{thm:main-general} as follows. Suppose $T$ be a locally finite geometric triangulation of $\R^2$ with symmetric weights $c_{pq}=c_{qp}\ge 0$ assigned to its edges. Suppose that the graph $G$ formed by the edges with $c_{pq}>0$ is connected and  $(G, c)$ is in equilibrium stress $c$, then the network $(G,c)$ has the Liouville property.

Our result also implies the Liouville property on Delaunay triangulations with cotangent weights, a class of networks extensively studied in discrete conformal geometry. 
Recall that a geometric triangulation $T$ of $\R^2$ is \emph{Delaunay} if the open circumdisk $D$ of every triangle contains no vertex of $T$. Equivalently, for every edge of $T$, the sum of the two angles opposite  that edge is at most $\pi$. % Vertices are allowed to lie on the circumcircle $\partial D$. %This allows arbitrary triangulations of polygons whose vertices lie on one empty circle. 
Delaunay triangulations are the most commonly used tools in computational geometry \cite{devadoss2025discrete}.

 The finite-element approximation of the classical Laplacian $\frac{\partial^2}{\partial x^2} +\frac{\partial ^2}{\partial y^2}$  using a triangulation $T$ in the plane is the discrete cotangent Laplacian introduced by Duffin \cite{duffin1959distributed} on the 1-skeleton $T^{(1)}$ of $T$ with  weight  $w_{pq}$ at the edge $\{p,q\}$ \begin{equation}\label{eq:weights}
 w_{pq}=w_{qp}=\frac12\bigl(\cot\alpha_{pq}+\cot\beta_{pq}\bigr). 
\end{equation}
Here $\alpha_{pq}$ and $\beta_{pq}$ denote the angles opposite the edge $\{p,q\}$ in the triangulation. The weights $w_{pq}$ are non-negative due to the Delaunay condition that $\alpha_{pq}+\beta_{pq} \leq \pi$, and are commonly referred to as
the \emph{cotangent weight}. The corresponding discrete Laplacian, known as the cotangent Laplacian, is defined by 
\begin{equation}\label{eq:laplacian}
 \Delta_wu(p)=\sum_{q:\,\{p,q\}\in E}w_{pq}\bigl(u(q)-u(p)\bigr), 
\end{equation} for any $u: V(T) \to \R$.  See \cite{Pinkall-pothier}, \cite{Braess-finite-elements} for more details. Naturally, one would expect the Liouville property to hold for the cotangent Laplacian on Delaunay triangulations.

In the vertex-scaling discrete conformal geometry of surfaces \cite{luo}, the cotangent Laplacian is the Jacobian matrix of the map that sends the discrete conformal factor to the discrete curvature. The infinitesimal rigidity of ideal convex polyhedra in hyperbolic 3-space and the discrete uniformization of polyhedral surfaces  predicts that the Liouville property holds for cotangent-weight planar triangulations. 
This motivates us to investigate the Liouville property.  See also the work of Z. He \cite{he-rigidity} on Liouville property on nerves of planar circle packing.

The geometry of circumcircles of triangles shows that reciprocal dual equation holds for the cotangent weights, where we take the associated points $g_L$ and $g_R$ to be the centers of the circumdisks of triangles adjacent to the edge, i.e., $g_L-g_R =w_{pq} J(p-q)$.
See Figure \ref{fig:dual-weights}(b). 

\begin{figure}[htbp]
\centering
\input{fig_dual_weights.tikz}
\caption{Dual edges and geometric weights. (a) The reciprocal dual equation, with the edge oriented from $q$ to $p$, $L$ and $R$ its left and right faces, and $J(x,y)=(-y,x)$. The length ratio shown is for a positive weight. (b) Circumcenters give the cotangent weight, with $\alpha_{pq}$ and $\beta_{pq}$ opposite to the edge $pq$. (c) For a circle packing, the dual vertices are the incenters of the two triangles; $r_p,r_q$ are the packing radii and $\rho_L,\rho_R$ the triangle inradii, giving the Dubejko weight. 
}
\label{fig:dual-weights}
\end{figure}
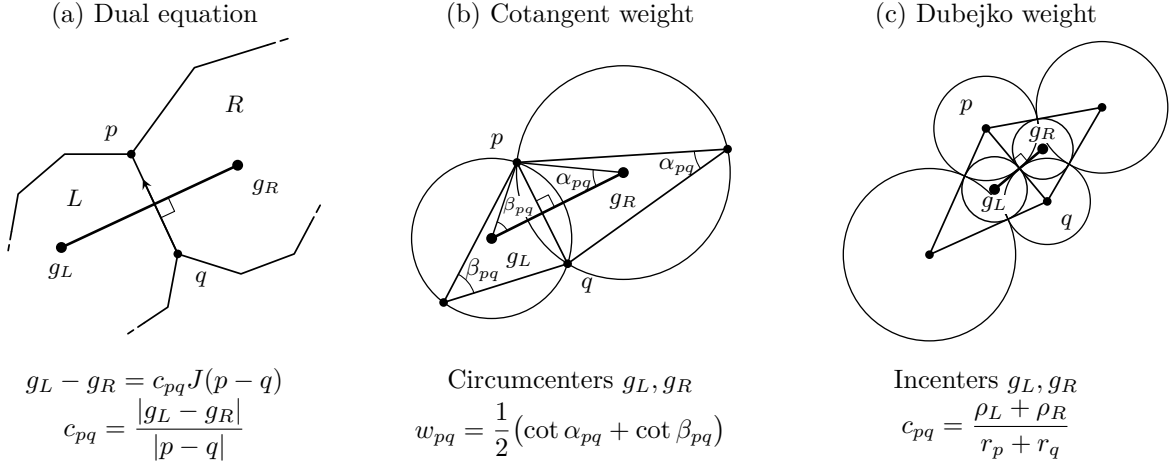

It follows from the reciprocal dual equation that the restriction of affine functions to the triangulation is discrete harmonic with respect to the cotangent Laplacian. Also the graph formed by edges with strictly positive cotangent weights is connected since it is the dual graph to the Voronoi cells of vertices $V=V(T)$. By Theorem \ref{thm:main-general}, we have,

\begin{corollary}\label{thm:main}
Let $T$ be a locally finite Delaunay triangulation of $\R^2$. Then every bounded discrete harmonic function is constant with respect to the cotangent Laplacian.
\end{corollary}
Dai--Wu~\cite[Proposition~4.3]{DaiWu} proved the above corollary under extra assumptions on $T$.

It turns out Corollary \ref{thm:main} is the tip of an iceberg of a whole family of geometric triangulations with weights coming from various different discrete conformal schemes for discretizing classical conformality. The cotangent Laplacian arises from the vertex-scaling discretization of conformal geometry; see \cite{luo, Schlenker2001, fillastre2008polyhedral,bobenko-pinkall-springborn, glsw}. The circle packing theory \cite{thurstons-notes}, \cite{stephenson-circle-packing-book} is another prominent discrete conformality. The associated Dubejko edge weight \cite{dubejko1997random} is constructed from Figure 1 (c), where the reciprocal dual equation holds. In \cite{glickenstein2011}, Glickenstein introduced a family of discretization schemes for 2-dimensional conformal structures, including vertex scaling and circle packing, and showed that each discrete conformal scheme in the family induces natural edge weights $c$ generalizing both cotangent weights and Dubejko weights on the 1-skeleton of triangulations of surfaces satisfying the reciprocal duality equation~\eqref{eq:dual}, which gives equation~\eqref{eq:es}.
Therefore, the Liouville property holds for all such triangulations  of the plane. See Section~\ref{sec:consequences} for detailed discussions.

A related question concerns recurrence of the network $(T^{(1)}, w)$, where $T^{(1)}$ is the 1-skeleton of the triangulation $T$ and $w$ is the cotangent weights. Recall that a network is said to be \it recurrent \rm if the random walk almost surely returns to its starting vertex infinitely often, and \it transient \rm otherwise \cite{BenSch1}.
It is well known that a recurrent network has the Liouville property, but the converse fails even for planar Delaunay triangulations with cotangent weights. 

\begin{theorem} \label{tran} If a Delaunay triangulation $T$ of the plane is such that the diameter of its Voronoi cells $C_v$ grows at most linearly in $|v|$ for vertices $v$ of $T$, then its 1-skeleton with the cotangent weights is recurrent.  Furthermore, 
there exists a Delaunay triangulation of the plane such that its 1-skeleton with the cotangent weights is transient.
\end{theorem}

\medskip
\noindent
 {\bf Organization.}  
% In Section 2, we recall the basics about the discrete Laplacian on networks. 
In Section~\ref{sec:geometric}, we recall the Adel'son-Vel'skii theorem and prove the main result, Theorem \ref{thm:main-general}.  Theorem \ref{tran} is proved in Section~\ref{sec:transient}.
In Section~\ref{sec:consequences}, we discuss several applications of Theorem \ref{thm:main-general} on discrete conformal schemes by verifying that affine functions are discrete harmonic.  In Appendix~\ref{sec:av-proof}, we reproduce a proof of the Adel'son-Vel'skii theorem.

\medskip
\noindent
 {\bf Acknowledgments.}  
  The work of F. L. is supported in part by NSF grants DMS 2220271 and DMS 2501286, and by Simons Foundation grant SFI-MPS-SFM-00011051. The work of Z. R. is partially supported in part by NSF grant DMS-2603932. The work of Y. L. is supported by National Natural Science Foundation of China under Grant No. 12601117. 

\medskip
\noindent
 {\bf Tool and computational resource disclosure.}  
The key idea of the proof of Theorem \ref{thm:main-general} is produced with the assistance of ChatGPT, which brought the link between discrete harmonic functions and the Adel'son-Vel'skii theorem.
The original question on the Liouville and recurrence properties of the cotangent Laplacian on planar graphs was posed by the authors. 
The authors also used ChatGPT to proofread the manuscript for grammar and language.

The authors take full responsibility for the content of the article.

\section{The Adel'son-Vel'skii theorem and a Proof of Theorem \ref{thm:main-general}}\label{sec:geometric}

We first  recall the Adel'son-Vel'skii theorem in this section. Since its proof is difficult to find in the literature, a proof is provided in the Appendix~\ref{sec:av-proof}. For a continuous function $U$ and an affine function $A$ on $\R^2$, a connected component of $\{U>A\}$ or $\{U <A\}$ in the plane can be viewed as the vertical projection of a portion of the graph of $U$ cut off by the plane $z=A(x)$. 
We use $u \cdot v$ and $\langle u, v \rangle $ to denote the dot product of two vectors $u, v$.

\begin{theorem}[Adel'son-Vel'skii~\cite{AV45}]\label{thm:AV}
Let $U:\R^2\to\R$ be continuous and have sublinear growth. Suppose that for every affine function $A:\R^2\to\R$, every connected component of $\{z\in\R^2:U(z)>A(z)\}$ and of $\{z\in\R^2:U(z)<A(z)\}$ is unbounded. Then there are a unit vector $a\in\R^2$ and a continuous function $f:\R\to\R$ such that $U(z)=f(a\cdot z)$ for all $z\in\R^2$.
\end{theorem}

The Adel'son-Vel'skii theorem can also be founded in Landis~\cite[Chapter~I, \S0, p.~5]{Landis1963}. It implies the classical Liouville theorem in the plane, using the maximum principle for harmonic functions. Note that Adel'son-Vel'skii's theorem does not hold in dimension three or higher.

We now prove Theorem \ref{thm:main-general}.  
 Call two distinct vertices $p,q$ in a graph $(V,E)$ \it neighbors \rm if they are the endpoints of the same edge.

 \begin{proof}[Proof of Theorem~\ref{thm:main-general}] 
We begin by showing that affine functions are discrete harmonic. 
\begin{lemma}\label{lem:affine-harmonic}
    Suppose $A: \R^2\to\R$ is an affine function. Then the restriction of $A$ to $V$ is discrete harmonic.
\end{lemma}

\begin{proof}
    The affine function $A$ has the form
    \begin{equation}
        A(x)= \langle x, a \rangle + k,
    \end{equation}
    for some unique $a\in\R^2$ and $k\in\R$. Therefore, for any vertex $p\in V$, by \eqref{eq:lap}, we have
    \begin{equation*}
        \Delta_cA|_V(p)=\sum_{q:\,\{p,q\}\in E}c_{pq}\bigl(A(q)-A(p)\bigr)
        =\sum_{q:\,\{p,q\}\in E}c_{pq}\langle q-p,a\rangle
        =\left\langle \sum_{q:\,\{p,q\}\in E}c_{pq}(q-p), a\right\rangle = 0.\qedhere
    \end{equation*}
\end{proof}

Let $H$ be a connected component of $G_+$ equipped with the restriction of $c$ on edges of $H$, $V(H)$ be the vertex set of $H$, and  $u:V(H)\to\mathbb{R}$ be a bounded  $c$-harmonic function.
%\textcolor{blue}{the terminology $c$-harmonic }\ZR{It's probably better to use ``$c$-harmonic'' when introducing it in the introduction, since we also use $a$-harmonic. Feng: Done}. 
A singleton component $H$ is trivial. So we may assume $H$ contains an edge. In this case, $H$ must be an infinite graph due to equation \eqref{eq:es}. Indeed, \eqref{eq:es} implies every vertex $p \in V(H)$ is in the convex hull of its neighboring vertices in $H$. The last property cannot hold for a finite set $V(H)$.  Adding a constant preserves
harmonicity and the desired conclusion, so we may assume $0\leq u\leq M$ for
some finite $M$.

In the first step, we choose averaging coefficients $a_{pq}\geq0$ for directed edges from $p$ to $q$ as follows. Note that $a_{pq}$ may be different from $a_{qp}$. 
For $p\in V(H)$, put
\[
  a_{pq}=\frac{c_{pq}}{\sum_{r :\{r, p\} \in E} c_{pr}}.
\]
For $p\notin V(H)$, choose $a_{pq}>0$ for every neighbor $q$ of $p$ so that
\begin{equation}\label{eq:averaging}
  \sum_{q: \{q,p\} \in E} a_{pq}=1,
  \qquad
  \sum_{q: \{q,p\} \in E} a_{pq}q=p.
\end{equation}
Such coefficients exist because $p$ lies in the interior of the convex hull
of its neighbors in $T^{(1)}$.  The identities \eqref{eq:averaging} also hold on $p \in V(H)$ by the
equilibrium assumption \eqref{eq:es}.  For $p \in V(H)$, by definition, every
neighbor $q$ with $a_{pq}>0$ belongs to $V(H)$, and harmonicity gives
$u(p)=\sum_{q: \{q,p\} \in E} a_{pq}u(q)$.

Note that, by construction, for all $p \in V$, 
\begin{equation} \label{eq:affine} \sum_{q: \{q,p\} \in E} a_{pq}(p-q)=0.
    \end{equation}
Indeed, if $p \in V(H)$, it follows from the definition of $H$ and $a_{pq}$, if $p \notin V(H)$, it follows from \eqref{eq:averaging}. 

 Having constructed the new positive edge weights $a_{pq}$, we now extend  $u$ to a discrete $a$-harmonic function $\widetilde{u}$ on $V$ by repeated weighted averaging.
Start with
\[
  u_0(p)=
  \begin{cases}
    u(p), & p\in V(H),\\
    0,    & p\notin V(H).
  \end{cases}
\]
Keep the values on $H$ fixed and, for $n=0,1,2,\ldots$, define
\begin{equation}\label{eq:iteration}
  u_{n+1}(p)=
  \begin{cases}
    u(p), & p\in V(H),\\[2pt]
    \displaystyle\sum_{q:\{q,p\}\in E} a_{pq}u_n(q), & p\notin V(H).
  \end{cases}
\end{equation}
Thus each new value outside $V(H)$ is the weighted average of the previous values
at neighboring vertices. The update is simultaneous.

Since $u_0\geq 0$, we have $u_1\geq u_0$. If $p \notin V(H)$, then $u_2(p) =\sum_{q: \{q,p\} \in E} a_{pq}u_1(q) \geq \sum_{q: \{q,p\} \in E} a_{pq} u_0(q) =u_1(p)$, i.e., $u_2 \geq u_1$. Since nonnegative weighted averages preserve
inequalities, so induction gives
\[
  0\leq u_0(p)\leq u_1(p)\leq u_2(p)\leq\cdots\leq M.
\]
The upper bound follows because the coefficients sum to one. Hence the
pointwise limit exists, and we define, for each $p \in V$, 
\begin{equation}\label{eq:limit}
  \widetilde{u}(p):=\lim_{n\to\infty}u_n(p).
\end{equation}
It satisfies $\widetilde{u}=u$ on $V(H)$ and $0\leq\widetilde{u}\leq M$.
For $p\notin V(H)$, pass to the limit in the finite sum \eqref{eq:iteration};
for $p\in V(H)$, use the harmonicity of $u$ and the fact that all
positive-coefficient neighbors lie in $V(H)$. Therefore
\begin{equation}\label{eq:harmonic}
  \widetilde{u}(p)=\sum_{q: \{q,p\} \in E} a_{pq}\widetilde{u}(q)
  \qquad\text{for every }p\in V.
\end{equation}

Let $U:\mathbb{R}^2\to\mathbb{R}$ be the continuous piecewise affine interpolant
of $\widetilde{u}$ on $T$. It is bounded because each value in a triangle is a
convex combination of its vertex values.

\begin{lemma}
\label{lem:no-caps}
Suppose $H$ is a connected component of the graph of positive $c$-weights and $u$ is a discrete $c$-harmonic function on $H$.  Then for every affine function $A$, neither $\{U>A\}$ nor $\{U<A\}$ has a bounded connected component.
\end{lemma}

\begin{proof}
Given a component $H$ of $G_+$, let $\widetilde{u}: V\to\R_{>0}$ be the function defined as above in terms of $u|_V(H)$. Set $w=\widetilde{u}-A|_V$, and let $W=U-A$ be its piecewise affine interpolant. Then $w$ is $a$-discrete harmonic by equations \eqref{eq:affine}, \eqref{eq:harmonic}, and Lemma \ref{lem:affine-harmonic}. Suppose otherwise that $O$ is a bounded connected component of $\{W>0\}$, which implies that $S:=V\cap O$ is a finite set.

It is immediate that $S$ is nonempty since  $W$ can not achieve a strict local maximum in the interior of a face or an edge. Indeed, take a point $z\in O$ and a closed triangle $T$ containing it. Then some vertex $p$ of $T$ has $W(p)\ge W(z)>0$. Since both $U$ and $A$ are affine on $T$, the line segment $[z,p]$ lies in $\{W>0\}$. Thus $p$ belongs to the same component $O$.

Let $M=\max_{p\in S}w(p)>0$, and let $F=\{p\in S:w(p)=M\}$. For $p\in F$, we consider its neighbors in $V$.
For a neighbor $q$ with $w(q)>0$, $W$ is positive along the edge connecting $p$ and $q$, which indicates that $q$ belongs to $O$. Thus we have $w(q)\le M$. Every other neighbor has $w(q)\le0<M$. Hence
$$
 0=\Delta_a w(p)=\sum_{q:\{q,p\}\in E} a_{pq}\bigl(w(q)-M\bigr)
$$
is a sum of nonpositive terms. Therefore, every positive-weight neighbor of $p$ must have value $M$ and lie in $F$.

We claim that $F\cap H \neq \varnothing$. Indeed, pick a $p\in F$ and a finite edge path from $p_0=p, p_1, ..., p_k$ with $p_k \in H$ and  $\{p_i, p_{i+1} \} \in E$.  Without loss of generality, we may assume that $p_i \notin H$ unless $i=k$. By construction $c_{p_i p_{i+1}}>0$. By the argument above, $c_{p_0p_1}>0$ implies $p_1 \in F$. Inductively, we see that all $p_i$ are in $F$. Therefore $p_k \in F \cap H$.

By the connectedness of $H$, the finite nonempty set $F$ must contain all vertices of $H$, which contradicts that $V(H)$ is an infinite set.

Thus $\{U>A\}$ has no bounded component. Applying the same argument to $-u$ and $-A$ proves the same assertion for $\{U<A\}$.
\end{proof}

By the Adel'son--Vel'skii theorem \ref{thm:AV} 
\[
  U(z)=f(\nu\cdot z)
\]
for a unit vector $\nu$ and a continuous function $f:\mathbb{R}\to\mathbb{R}$.
Local finiteness makes $f$ locally piecewise affine. If $f$ is affine,
boundedness makes it constant, and the proof is complete.

Otherwise, choose a breakpoint $t_0$ with adjacent slopes $m_-\neq m_+$,
and set
\[
  L=\{z:\nu\cdot z=t_0\},
  \qquad
  A_0(z)=f(t_0)+\frac{m_-+m_+}{2}(\nu\cdot z-t_0).
\]
The line $L$ lies in the one-skeleton of $T$: a face interior cannot meet a
line where $U$ is not affine. Moreover, $L$ contains a vertex, since local
finiteness forces an edge to lie along $L$.

For $p\in V\cap L$, affineness of $U$ along each incident edge gives
\[
  \widetilde{u}(q)-A_0(q)
  =\frac{m_+-m_-}{2}\lvert\nu\cdot q-t_0\rvert
  \qquad\text{for every neighbor }q.
\]
Using \eqref{eq:averaging}, \eqref{eq:harmonic}, and
$\widetilde{u}(p)=A_0(p)$, we obtain
\[
  0=\sum_q a_{pq}\bigl(\widetilde{u}(q)-A_0(q)\bigr)
  =\frac{m_+-m_-}{2}\sum_q a_{pq}\lvert\nu\cdot q-t_0\rvert.
\]
Since $m_+\neq m_-$, every neighbor $q$ of $p$ with $a_{pq}>0$ lies on $L$. Such a vertex
$p$ cannot lie outside $H$. Indeed, if otherwise, let  $q_1, ..., q_k$ be the set of all neighboring vertices of $p$.  Then by the construction, $a_{pq_i}>0$. It follows that all $q_i \in L$. But this shows the vertex star of $p$ is collinear, which contradicts the triangulation assumption. 

Since $a_{pq}>0$ for every positive-$c$ edge issuing from a vertex of $H$, starting from a vertex of $L$ and following positive-weight edges in the
connected component $H$ therefore gives $H\subset L$. Consequently,
$u\equiv f(t_0)$ on $H$, i.e., $u$ is constant  on $H$. This finishes the proof of 
Theorem \ref{thm:main-general}.
\end{proof}

\begin{remark} Note that if $G_+$ is connected and the edge lengths of the triangulation $T$ are uniformly bounded, Theorem \ref{thm:main-general} can be strengthened to assert that every discrete harmonic function on $V$ with sublinear growth is constant.
Here a function with sublinear growth $f$ on $V$ means $\lim_{ |v| \to \infty} |f(v)|/|v| =0$. Indeed, in this case, the sublinear growth of $u$ implies that its affine interpolant $U$ on $\R^2$ also has sublinear growth. 
\end{remark}

\section{Recurrence and Transience of Delaunay triangulations}\label{sec:transient}

In this section, we prove Theorem \ref{tran}, which concerns recurrence for planar Delaunay triangulations. We show that if the diameters of Voronoi cells grow at most linearly in terms of the distances from their centers to the origin, then the 1-skeleton with cotangent weights is recurrent. Note that Delaunay triangulations of uniformly bounded edge lengths have Voronoi cells with uniformly bounded diameters, and hence satisfy the linear growth condition.  On the other hand, we construct a planar Delaunay triangulation whose cotangent-weighted network is transient.

Recall that an infinite network is recurrent if and only if the effective conductance (capacity) from a vertex to infinity is zero; equivalently, the effective resistance from that vertex to infinity is infinite. For a network with the conductance, i.e., the weight, $c$, define the Dirichlet energy of a function $f: V\to \R$ to be  
$$
\mathcal E(f)=\sum_{\{p,q\}\in E} c_{pq}\big(f(p)-f(q)\big)^2.
$$
The effective conductance from a vertex $a$ to infinity is

\begin{equation}
\label{conductance}
    C_{\mathrm{eff}}(a,\infty)
=
\inf\Big\{\mathcal E(f):\ f(a)=1,\ f \text{ has finite support}\Big\}.
\end{equation}

The effective resistance is its reciprocal:
$
R_{\mathrm{eff}}(a,b)
=
\frac{1}{C_{\mathrm{eff}}(a,b)}
$. Another equivalent formulation of the effective
resistance between vertices $a$ and $\infty$ is
\begin{equation}
\label{resistence}
    R_{\mathrm{eff}}(a,\infty)
=
\inf\Bigg\{
\sum_{e\in E}  \frac{i(e)^2}{c(e)}
:\ 
i \text{ is a unit flow from } a \text{ to } \infty
\Bigg\}.
\end{equation}
See \cite{doyle1984random} for a comprehensive treatment on this topic.

We first prove the following theorem.

\begin{recurrence}
    Suppose a Delaunay triangulation $T$ of the plane is such that the diameter of its Voronoi cells $C_v$ grows at most linearly in $|v|$ for vertices $v$ of $T$, namely, there exists $A>0$ such that $\diam(C_v)\leq A(1 + |v|)$.
    
    Then its 1-skeleton with the cotangent weights is recurrent.  Furthermore, there exists a Delaunay triangulation of the plane such that its 1-skeleton with the cotangent weights is transient.
 
\end{recurrence}

It has been shown in \cite{he-rigidity, he1995hyperbolic, GeYuZhou2025, GGS} that a locally finite circle packing in the plane with the weight from dual equations in Figure~\ref{fig:dual-weights} and under certain conditions on intersection angles of circles is always recurrent. As we will see below, the recurrence property does not hold for the general cotangent weighted Delaunay triangulation.

\begin{proof}
    Since the growth hypothesis is unchanged by translation, with a possibly different $A$, we may choose a vertex $o$ as the origin.
    By equation \eqref{conductance}, it suffices to find a sequence of finitely supported functions $f_n$ with $f_n(o)=1$, such that
    \begin{equation}
        \mathcal {E}(f_n):=\sum_{\{u,v\}\in E}w_{uv}\bigl(f_n(u)-f_n(v)\bigr)^2\to 0,
    \end{equation}
    as $n\to\infty$. Here each edge in this sum is counted once.

    We first compute the area of Voronoi cells in terms of the distance between vertices.
    Let $K=1+A$. For every $v\in V$, since $v\in C_v$, we have for all $x\in C_v$,
    \begin{equation}\label{eq:radii}
        1+|x|\le 1+(|v|+\diam(C_v))\le K(1+|v|).
    \end{equation}
    Let $s_{uv}$ be the length of the common side $uv$ of cells $C_u$ and $C_v$. Then the cotangent weight $w_{uv}$ is given by (see See Figure \ref{fig:dual-weights}(b)), 
    \begin{equation}
        w_{uv}=\frac{s_{uv}}{|u-v|}.
    \end{equation}
    Indeed, the signed distances of the two adjacent circumcenters from the midpoint of $uv$, measured toward their respective opposite vertices, are $\frac12|u-v|\cot\alpha_{uv}$ and $\frac12|u-v|\cot\beta_{uv}$. Hence their sum is $s_{uv}$.

    Drop a perpendiculars from $v$ to each side of $C_v$. The triangle over the side shared with $C_u$ has base $s_{uv}$ and height $|u-v|/2$; see Figure~\ref{fig:voronoi-triangle-area}. Summing its area over the sides gives
    \begin{equation}\label{eq:area}
        \area(C_v)=\sum_{\{u,v\}\in E}\frac{s_{uv}|u-v|}{4}=\frac14\sum_{\{u,v\}\in E}w_{uv}|u-v|^2.
    \end{equation}

\begin{figure}[htbp]
\centering
\begin{tikzpicture}[scale=1.3]
 \coordinate (v) at (0,0);
 \coordinate (u) at (2.7,0);
 \coordinate (a) at (1.35,1.05);
 \coordinate (b) at (1.35,-1.1);

 % Triangle shading
 \fill[blue!12] (v)--(a)--(b)--cycle;

 % Boundary of C_v, without fill
 \draw[black,thin]
   (-1.1,-0.68)--(-0.68,1.27)--(0.42,1.53)
   --(a)--(b)--(0.16,-1.45)--cycle;

 \draw[black!60,thin] (v)--(a) (v)--(b);
 \draw[red,thick] (a)--(b);
 \draw[dashed,black!60] (v)--(u);

 % Right-angle marker
 \draw[black,thin]
   (1.19,0)--(1.19,0.16)--(1.35,0.16);

 \fill[black] (v) circle (1.6pt) (u) circle (1.6pt);

 % Labels
 \node[above left,inner sep=3pt] at (v) {$v$};
 \node[above,inner sep=3pt] at (u) {$u$};
 \node at (-0.45,0.95) {$C_v$};
 \node[right=4pt,red] at (1.35,0.68) {$s_{uv}$};
 \node[below=2pt,black!70,inner sep=0pt,font=\footnotesize]
   at (0.82,0) {$|u-v|/2$};
\end{tikzpicture}
\caption{The shaded triangle has area
$\area=\frac{s_{uv}|u-v|}{4}
=\frac{w_{uv}|u-v|^2}{4}$.}
\label{fig:voronoi-triangle-area}
\end{figure}
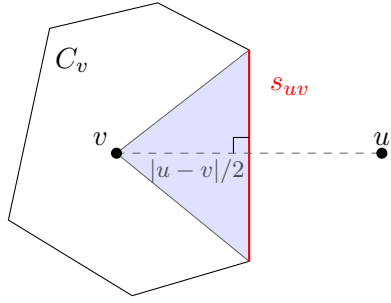

Now for integer $n>2$, we define
\begin{equation}\label{eq:cutoff}
    f_n(v)=\left(1-\frac{\log(1+|v|)}{\log n}\right)_+.
\end{equation}
Here $t_+:=\max\{t,0\}$. Then $f_n(o)=1$ and $f_n(v)=0$ when $|v|\ge n-1$, so $f_n$ is finitely supported by local finiteness of $T$. If $|v|\leq |u|$, the mean value theorem gives
\begin{equation}\label{eq:gradient}
    |f_n(u)-f_n(v)|\le\frac{|u-v|}{(1+|v|)\log n},
\end{equation}
and hence
\begin{equation}
    w_{uv}\bigl(f_n(u)-f_n(v)\bigr)^2\leq \frac{w_{uv}|u-v|^2}{(1+|v|)^2(\log n)^2}.
\end{equation}
 By \eqref{eq:area}, assign each unordered
edge to one endpoint of smaller radius, then we have
\begin{equation}\label{eq:energy}
\begin{aligned}
    \mathcal E(f_n)&\le
    \sum_{|v|<n-1} \sum_{\{u,v\}\in E}w_{uv}\bigl(f_n(u)-f_n(v)\bigr)^2\\
    &\leq
    \frac{1}{(\log n)^2}\sum_{|v|<n-1}
    \frac{\sum_{\{u,v\}\in E}w_{uv}|u-v|^2}{(1+|v|)^2}
    =
    \frac{4}{(\log n)^2}\sum_{|v|<n-1}
    \frac{\area(C_v)}{(1+|v|)^2}.
\end{aligned}
\end{equation}
For $x\in C_v$, \eqref{eq:radii} gives
\begin{equation}
    \frac{\area(C_v)}{(1+|v|)^2}  = \int_{C_v}\frac{dA}{1+ |v|^2}\le K^2\int_{C_v}\frac{dA}{(1+|x|)^2}.
\end{equation}
Note that the cells in \eqref{eq:energy} have disjoint interiors. Moreover, they all lie in
$B(0,Kn)$ because $|v|<n-1$ and
$1+|x|\le K(1+|v|)<Kn$ when $x\in C_v$. Consequently,
 $$\mathcal E(f_n)
 \le \frac{4K^2}{(\log n)^2}
       \int_{B(0,Kn)}\frac{dA}{(1+|x|)^2}
 =\frac{8\pi K^2}{(\log n)^2}
       \int_0^{Kn}\frac{r\,dr}{(1+r)^2}
 \le \frac{8\pi K^2\log(1+Kn)}{(\log n)^2}
 \to 0$$
as $n\to\infty$.
This proves the recurrence.

Finally, we now construct a planar Delaunay triangulation using a sequence of nested rectangles whose cotangent-weighted network is transient, in contrast to the above recurrent case.

Consider the following four sets of points $A_n = \{(4^n, 2^n)\}$, $B_n = \{(4^n, -2^n)\}$,  $C_n = \{(-4^n, -2^n)\}$, and $D_n = \{(-4^n, 2^n)\}$ for $n = 0, 1, \cdots$ together with the origin $O = (0, 0)$. (They lie on the curves $(\pm x, \pm \sqrt{|x|})$.) Let $V = \{A_n, B_n, C_n, D_n, O\}$. Then it is straightforward to check that the Delaunay triangulation $T$ determined $V$ is given by the following picture. Notice that its Delaunay tesselation consists of only four triangles, where other cells are trapezoids. Between consecutive rectangles there are four isosceles, hence cyclic, trapezoids, and their dashed diagonals have cotangent weight $0$.
\input{fig_transient}

To show that the network $(T, w)$ with the cotangent weight $w$ on $T$ is transient, we only need to show that the effective resistance from $O$ to the infinity is a finite number. 
By equation \eqref{resistence},  the network $(T, w)$ is transient if  there exists an infinite edge path $v_0v_1, v_1v_2, ..., v_{n}v_{n+1}, ...$ starting from a vertex $v_0$ and escaping all compact sets such that $\sum_{n} \frac{1}{w_{v_nv_{n+1}}} $ is finite by taking the unit flow $i(v_kv_{k+1}) = 1$ for $k = 0, 1, \cdots$.

Therefore, we only need to show that the effective resistance in series of edges $OA_0$ and $A_kA_{k+1}$ for $k = 0, 1, 2, \cdots$ is a finite number. The weight on $OA_0$ is $w_{01} = 1/2(\cot (\pi/4) +\cot (\pi/4)) = 1$ since the middle rectangle is a square. The cotangent weight $w_{k(k+1)}$ on $A_kA_{k+1}$ can be computed by 
$$w_{k(k+1)} = \frac{1}{2}(\cot \angle A_kB_{k+1}A_{k+1} + \cot \angle A_kD_{k+1}A_{k+1})$$
with the triangulation shown in the picture. To compute $\cot \angle A_kB_{k+1}A_{k+1}$, consider the vectors $u_k = \overrightarrow{B_{k+1}A_k}$ and $v_k = \overrightarrow{B_{k+1}A_{k+1}}$ in the $xy$-plane in  $\R^3$, then 
$$\cos \angle A_kB_{k+1}A_{k+1} = \frac{u_k\cdot v_k}{|u_k||v_k|}, \quad \sin \angle A_kB_{k+1}A_{k+1} = \frac{|u_k\times v_k|}{|u_k||v_k|} = \frac{|\det (u_k ,v_k)|}{|u_k||v_k|},$$ 
so
$$\cot \angle A_kB_{k+1}A_{k+1} = \frac{u_k\cdot v_k}{|\det (u_k ,v_k)|}.$$
Then restricted to $\R^2$, $u_k = 3(-4^{k}, 2^{k})$ and $v_k = (0, {2^{k+2}})$. Therefore, 
$$\cot \angle A_kB_{k+1}A_{k+1} = \frac{1}{2^{k}}.$$
Similarly, define the vectors $f_k = \overrightarrow{D_{k+1}A_k} = (5\cdot 4^k, -2^{k})$ and $g_k = \overrightarrow{D_{k+1}A_{k+1}} = (2\cdot 4^{k+1},0)$ in $\R^3$, we have 
$$\cot \angle A_kD_{k+1}A_{k+1} =  \frac{f_k\cdot g_k}{|\det (f_k ,g_k)|} = 5\cdot 2^{k}.$$
Therefore, $w(O, A_0) = 1$ and for $k\geq 0$, we have 
$$w_{k(k+1)} = \frac{1}{2}(5 \cdot 2^k + \frac{1}{{2^{k}}}) \geq 2^k.$$
Then the effective resistance in the path from $O$ along $A_k$ to the infinity is bounded by 
$$ R_{\mathrm{eff}}(O,\infty) \leq  \frac{1}{w(O, A_0)} + \sum_{k=0}^{\infty}\frac{1}{w_{k(k+1)}} \leq  1 + \sum_{k=0}^{\infty}\frac{1}{2^k}<\infty.$$
Therefore, the network $(\mathcal{T}, w)$ is transient. 
\end{proof}

% \newpage
\section{Consequences of Theorem \ref{thm:main-general} in discrete conformal geometry}\label{sec:consequences}
 Thurston introduced circle packings \cite{thurstons-notes} as a discretization scheme of 2-dimensional conformal structure in 1978. Since then, various discrete conformal structures of polyhedral surfaces have been proposed \cite{BowersSt, zhou2017circle, luo}. Glickenstein~\cite{glickenstein2011} and Glickenstein--Thomas~\cite{glickenstein-thomas2017} first introduced the duality structure of discrete metrics so as to unify all previous discrete conformal schemes of Euclidean polyhedral surfaces into a continuous family. Zhang--Guo--Zeng--Luo--Yau--Gu~\cite{zhang2014} then provides the geometric interpretation of these general discrete conformal structures in convex geometry of hyperbolic 3-space. See also Xu--Zheng~\cite{xu-zheng2024} for the classification of discrete conformal structures with different spherical and hyperbolic background metrics. In this section, we briefly recall the duality of discrete conformal structures in \cite{glickenstein-thomas2017}, which implies that the reciprocal dual equation holds for all discrete conformal structures.

Let  $(S,T)$ denote a connected oriented triangulated surface $S$
with a triangulation $T$. Recall that $V,E,F$ be the set of vertices, edges, and faces of $T$ respectively. 
For simplicity, we write $i\in V$ for a vertex, $ij\in E$ for an edge joining vertices $i$ and $j$, and $ijk\in F$ for a triangle with vertices $i$, $j$, and $k$.
For functions defined on $V$, $E$, and $F$, respectively, we use the shorthand notation $f_i=f(i)$, $f_{ij}=f(\{i,j\})$, and $f_{ijk}=f(\{i,j,k\})$.

\begin{definition}
   A Euclidean polyhedral surface $(S,T,l)$ 
is a triangulated surface $(S,T)$ with a map
$l:E\to\mathbb{R}_{>0}$ satisfying the strict triangle inequalities for any face $\{i,j,k\}\in F$.
% , $l_{ij}< l_{ik} + l_{jk}.$
\end{definition}
The function $l$ is a discrete metric on $(S, T)$. A discrete conformal structure consists of a set of functions on vertices determining a set of discrete metrics as follows. 
\begin{definition}[\cite{glickenstein-thomas2017}]
Let  $\varepsilon:V\to\{-1,0,1\}$, $\eta:E\to\mathbb{R}$ be two
weights defined on the vertices and edges respectively of $(S, T)$. A Euclidean discrete
conformal structure on the weighted triangulated surface
$(S,T ,\varepsilon,\eta)$ is the set of maps $r:V \to \R_{>0}$  such that the edge length function $l:E\to \mathbb{R}_{>0}$ is a discrete metric satisfying
\begin{equation}
l_{ij}^2=\varepsilon_i r_i^2+ \varepsilon_j r_j^2+2\eta_{ij}r_ir_j,
\end{equation}
\end{definition}
This definition contains the previous discrete conformal structures as special cases. The following table from \cite{zhang2014} and \cite{xu2021rigidity} shows their correspondence.  
\begin{table}[htbp]
\centering
\renewcommand{\arraystretch}{1.15}
\resizebox{0.8\linewidth}{!}{%
\begin{tabular}{|lccc|}
\hline
Scheme
& $\varepsilon_i$
& $\varepsilon_j$
& $\eta_{ij}$ \\
\hline
Tangential circle packings
& $+1$ & $+1$ & $+1$ \\

Circle packings with intersection angles
& $+1$ & $+1$ & $(-1,1]$ \\

Inversive distance circle packings
& $+1$ & $+1$ & $(-1,+\infty)$ \\

Vertex scalings
& $0$ & $0$ & $(0,+\infty)$ \\

General discrete conformal structures
& $\{+1,0,-1\}$ & $\{+1,0,-1\}$ & $(-1,+\infty)$ \\
\hline
\end{tabular}%
}
\end{table}

For simplicity, we consider geometric triangulations $T$ in the plane whose discrete metric $l_{ij} = |p_i - p_j|$ is determined by some discrete conformal structure $(\R^2, T, \epsilon, \eta)$, where $p:V\to \R^2$ denote the position of vertices in $T$. It determines a natural reciprocal dual structure by the power distances to the (virtual) circles assigned to vertices as follows. For each face $\{i,j,k\}\in F$ of $T$, define a point $z_{ijk}\in \R^2$ satisfying 
$$|z_{ijk} - p_i|^2 - \epsilon_ir_i^2 = |z_{ijk} - p_j|^2 - \epsilon_jr_j^2  = |z_{ijk} - p_k|^2 - \epsilon_kr_k^2.$$
It is straightforward to check that the solution exists since the triangle is not degenerate. Geometrically, the weight $(\epsilon, \eta)$ determines the type of circles at vertices with $\epsilon$ and relations between circles with $\eta$. In particular, if $\epsilon_i=1$, the vertex $i$ is assigned with a circle with radius $r_i$; if $\epsilon_i=0$, the vertex $i$ is assigned with a point, or a degenerate circle; if $\epsilon_i=-1$, the vertex $i$ is assigned with a circle with imaginary radius $r_i$. The weight $\eta_{ij}$ records the inversive distance between the circles assigned to $i$ and $j$ when $\epsilon_i = \epsilon_j = 1$, namely, 
$$\eta_{ij} = \frac{l_{ij}^2 - r_i^2 - r_j^2 }{2r_ir_j}.$$
Meanwhile, point $z_{ijk}$ defines a reciprocal dual structure to $p_i$. Indeed, one can show that $z_{ijk} - z_{ijm}$ is perpendicular to $p_i - p_j$ as vectors in $\R^2$. The dual structure defines a weight $c_{ij}$ as in equation \ref{eq:dual}
\begin{equation}\label{eq:dual2} z_{ijk} -z_{ijm} = c_{ij}J (p_i-p_j),
\end{equation}
where $z_{ijk}$ (resp. $z_{ijm}$) is the point associated with the triangle to the left (resp. right) of the direct edge from $p_j$ to $p_i$. 
See the important work of Glickenstein~\cite{glickenstein2011} for more details on the duality structure. One can derive that $c_{ij}\geq 0$ when the triangulation $T$ is a weighted Delaunay triangulation and the corresponding $G_+$ is connected. See \cite{xu2021rigidity, bobenko2024decorated} for the definition. 

\begin{figure}[h!]
\centering
\input{mix1.tikz}
\caption{(a) Dual center of a weighted triangle; (b) Dual structure across two adjacent triangles.}
\end{figure}

 Theorem \ref{thm:main-general} implies that a generalization of Corollary \ref{thm:main}
\begin{corollary} \label{cor:2}
Let $T$ be a locally finite weighted Delaunay triangulation of the plane such that its discrete metric $l_{ij} = |p_i-p_j|$ is determined by a function $r: V\to \R_{>0}$ and a discrete conformal structure $(\R^2, T, \epsilon, \eta)$. Let $G_+=(V(T), E_+)$ with $E_+=\{ \{p,q\} \in E: c_{pq}>0\}$ be the subgraph of the 1-skeleton of $T$ with positive edge weights. Then every bounded discrete harmonic function is constant in each component of $G_+$ with respect to the weight $c_{ij}$ defined from the dual structure.

\end{corollary}
Notice that it is not clear whether a triangulation of the plane with an arbitrary prescribed weight exits, which begs a discrete uniformization theorem for general discrete conformal scheme. In some special cases of circle packings with intersection angles under specific conditions, a discrete uniformization theorem has been established by \cite{he-rigidity, GeYuZhou2025, zhao2026weyl} for example.

\begin{remark}
 Theorem~\ref{tran} can be extended to the general discrete conformal schemes discussed above under the additional assumption that the vertices $v$'s are contained in the corresponding power diagrams and the diameters of the power diagrams grow at most linearly in terms of $|v|$. The proof follows the same idea as that of Theorem~\ref{tran} and is therefore omitted.
\end{remark}

\appendix

% \newpage
\section{Proof of Theorem~\ref{thm:AV}}\label{sec:av-proof}

In this appendix, we prove Theorem~\ref{thm:AV}

\theoremstyle{plain}
\newtheorem*{AV}{Theorem~\ref{thm:AV}}
\begin{AV}[Adel'son--Vel'skii]
Let $U:\R^2\to\R$ be continuous and have sublinear growth. Suppose that for every affine function $A:\R^2\to\R$, every connected component of $\{z\in\R^2:U(z)>A(z)\}$ and of $\{z\in\R^2:U(z)<A(z)\}$ is unbounded. Then there is a unit vector $a\in\R^2$ and a continuous function $f:\R\to\R$ such that $U(z)=f(a\cdot z)$ for all $z\in\R^2$.
\end{AV}

%\newpage
For simplicity, we denote by (P) the hypothesis that for every affine function $A:\R^2\to\R$, every connected component of $\{z\in\R^2:U(z)>A(z)\}$ and of $\{z\in\R^2:U(z)<A(z)\}$ is unbounded. 
Note that if $U$ satisfies (P) (or has sublinear growth), then so do $-U$ and $U\circ\Phi$ for every rigid motion $\Phi$ of $\R^2$.

We prove the theorem with the following lemmas.
The first lemma shows that the hypothesis (P) yields the following maximum principle.

\begin{lemma}\label{lem:comparison}
    Suppose $U$ satisfies $\mathrm{(P)}$. Let $A$ be affine and $D\subset\R^2$ a nonempty bounded open set. Then 
    \begin{equation}
    \begin{aligned}
        \sup_D (U-A) &\leq \max_{\partial D} (U-A),\\
        \inf_D (U-A) &\geq \min_{\partial D} (U-A).
    \end{aligned}
    \end{equation}
\end{lemma}

\begin{proof}
    Since $D$ is bounded and nonempty, we know $\partial D$ is compact and nonempty. Let $M = \max_{\partial D} (U-A)$, and suppose otherwise that $\sup_D (U-A) > M$. Then the maximum of $U-A$ over the compact set $\overline{D}$, the closure of $D$, is realized at some $z_0\in D$, and moreover, $(U-A)(z_0)>M$.

    Choose $c$ with $M<c<(U-A)(z_0)$, and consider the component $Z$ of $\{ U > A+c\}$ that contains $z_0$. Then $z_0\in Z\cap D$ and $Z\cap \partial D=\varnothing$. Thus $Z\subset D$, so $Z$ is bounded, which contradicts $\mathrm{(P)}$ for the affine function $A+c$.

    Applied to $-U$ and $-A$, the same argument gives the second inequality.
\end{proof}

The second lemma provides a non-horizontal plane that splits $\{U>A\}$ or $\{U<A\}$ into at least two components.

%\newpage

\begin{lemma}\label{lem:disconnected}
    Suppose $U$ has sublinear growth and satisfies $\mathrm{(P)}$. If $U$ is not constant, then there is an affine function $A$ with non-zero gradient such that $\{U>A\}$ is disconnected.
\end{lemma}

\begin{proof}
    Since $U$ is not constant, after applying a rigid motion of the plane, we can choose point $P=(0,0)$ and $Q=(d,0)$ with some $d>0$ such that $U(P)<U(Q)$. Choose $\alpha>0$ and $c$ such that $U(P)<c$ and $U(Q)>\alpha d+c$.
    For $t\in\R$, define
    \begin{equation}
        A_t(x,y)=\alpha x+ty+c,
    \end{equation}
    and let $G_t=\{U>A_t\}$.

    Since $U$ has sublinear growth and $\alpha>0$, there exist $R_0>0$ sufficiently large such that $U(-R_0,0)>-\alpha R_0+c$. Let $B=(-R_0,0)$; then $B,Q\in G_t$ and $P\in\{U<A_t\}$ for every $t$.
    Now we prove that for some $t$, $B$ and $Q$ belong to different components of $G_t$, which then proves the statement. Suppose otherwise; then for every $t$, there is a path $\gamma_t\subset G_t$ from $B$ to $Q$.

    Note that any closed curve in $G_t$ has winding number $0$ about $P$. Indeed, by $\mathrm{(P)}$, the component $Z$ of $\{U<A_t\}$ containing $P$ is unbounded and disjoint from any curve in $G_t$. 
    % so each component of $\R^2\backslash \overline{Z}$ is simply connected. 
    Fix a reference path $\beta$ from $Q$ to $B$ avoiding $P$, and let
    \begin{equation}
        k(t)=\wind(\gamma_t* \beta,P).
    \end{equation}
    Then $k(t)$ is independent of the choice of $\gamma_t$ for each $t$. Here $*$ denotes
    concatenation.
    Since the fixed compact path $\gamma_t$ in $G_t$ remains in $G_s$ for all $s$ sufficiently close to $t$, we know that $k(t)$ is locally constant. Therefore $k$ is constant on $\R$.

    For sufficiently small positive number $\epsilon$, consider the path which is the union of following line segments
    \begin{equation}
        \gamma_-=\{(-R_0,y):-\epsilon\leq y\leq 0\}\cup\{(x,-\epsilon):-R_0\leq x\leq d\}\cup\{(d,y):-\epsilon\leq y\leq 0\}.
    \end{equation}
    Then $\gamma_-$ is a path from $B$ to $Q$.
    For all sufficiently large positive $t$, $\gamma_-$ lies in $G_t$. 
    Similarly, for all sufficiently large negative $t$, there is a path $\gamma_+$ just above the segment $[B,Q]$ by reflecting $\gamma_-$ about the $x$-axis  lies in $G_t$. 
    Since $\gamma_-$ and $\gamma_+$ together bound a rectangle containing $P$, their corresponding winding numbers are different, after the concatenation with $\beta$, which contradicts that $k(t)$ is constant.
\end{proof}

The next lemma compares the behavior of functions with sublinear growth and affine functions near infinity.

\begin{lemma}\label{lem:sublinear-property}
    Suppose $U$ is function with sublinear growth, and $A(x,y)=px+q$ with $p>0$. For every $c>0$, there exists $R(c)>0$ such that for all $z=(x,y)$ with $|z|>R(c)$, we have
    \begin{enumerate}
        \item if $x\geq c|y|$, then $U(z)<A(z)$;
        \item if $x\leq -c|y|$, then $U(z)>A(z)$.
    \end{enumerate}
\end{lemma}

\begin{proof}
    Let $k=c/\sqrt{1+c^2}$. When $x\geq c|y|$, we have $|z|^2=x^2+y^2\leq(1+c^2)x^2/c^2$. Thus $x\geq k|z|$.
    Since $U$ has sublinear growth, there exists $R_0>0$ such that $|U(z)|\leq pk|z|/2$ for $|z|\geq R_0$. Let $R(c)=\max\{R_0,\frac{2|q|}{pk}+1\}$. Then for $|z|\geq R(c)$ and $x\geq c|y|$,
    \begin{equation}
        A(z)\geq pk|z|-|q|>\frac{pk|z|}{2} \geq U(z).
    \end{equation}
    Similarly, if $x\leq -c|y|$, we have $U(z)>A(z)$. 
\end{proof}

\begin{corollary}\label{cor:component}
    Suppose $U$ has sublinear growth, and $A(z)=px+q$ with $p>0$. If $\{U>A\}$ is disconnected, then it has a component $\Omega$ such that
    \begin{equation}\label{eq:Omega}
        \Omega\subset \{z:|x|<c|y|\} \cup B(0,R(c)),
    \end{equation}
    for every $c>0$.
\end{corollary}

\begin{proof}
    For $c>0$, let $K(c)=\{z:x\leq -c|y|, |z|\geq R(c)\}$, which is connected. By Lemma~\ref{lem:sublinear-property}, $K(c)\subset \{U>A\}$. Since any two $K(c_1)$ and $K(c_2)$ contain the point $(-x,0)$ for sufficiently large $x$, we know that all such $K(c)$ lie in the same component $\Omega_0$ of $\{U>A\}$. Since $\{U>A\}$ is disconnected, it has a component $\Omega\neq\Omega_0$. Thus $\Omega\cap K(c)=\varnothing$.

    Similarly, let $D(c) = \{z:x\geq c|y|, |z|\geq R(c)\}$, by Lemma~\ref{lem:sublinear-property} again, $U(z)<A(z)$ if $z\in D(c)$, so $z\notin\Omega$. $\Omega\cap D(c)=\varnothing$. Hence, $\Omega$ is in the intersection of the complement of $K(c)$ and the complement of $D(c)$, which implies equation \eqref{eq:Omega}.
\end{proof}

We recall a simple result for convex functions, and the proof is left to the reader.

\begin{lemma}\label{lem:convex}
    Let $f$ be a convex function on an open interval $J=(a,b)$, and suppose that for every $\epsilon>0$, $f(x)\leq \epsilon|x|$ for all $x\in J$ with $|x|$ sufficiently large. If $b=\infty$, then f is non-increasing; if $a=-\infty$, then $f$ is non-decreasing; hence $f$ is constant if $J=\R$.
\end{lemma} 
% \feng{maybe modify the above lemma?}

We then prove the following core lemma of this section, which claims that under the assumption (P), if $\{U>A\}$ has a thin component $\Omega$, then $\Omega$ contains a complete line, on which $U$ is constant. For simplicity, such a line is called a \emph{level line}.

\begin{lemma}\label{lem:level-line}
    Suppose $U$ satisfies $\mathrm{(P)}$, and $A(z)=px+q$ with $p,q\in\R$. Let $\Omega$ be a component of $\{U>A\}$. If
    \begin{enumerate}
        \item for every $T>0$, $\Omega\cap (\R\times[-T,T])$ is bounded;
        \item for every $\epsilon>0$, there is $T>0$ such that if $z=(x,y)\in\Omega$ and $|y|\geq T$, then $|x|\leq\epsilon|y|$ and $U(z)-A(z)\leq \epsilon|y|$.
    \end{enumerate}
    Then there exists $t_0\in\R$ such that $\{x=t_0\}\subset \Omega$ is a level line.
\end{lemma}

\begin{proof}
    For simplicity, let $W=U-A$. Then $W>0$ on $\Omega$ and $W=0$ on $\partial \Omega$. By the hypothesis (P), $\Omega$ is unbounded. Let $I$ be image of the horizontal projection of $\Omega$ onto the $y$-axis in $\R^2$. Then $I$ is an open interval, and is unbounded; otherwise, $I$ is contained in some $[-T,T]$, and Condition~(1) forces $\Omega$ to be bounded.

    For $t\in I$, let $h(t)=\sup\{W(s,t):(s,t)\in\Omega\}$. Then $h(t)$ is positive, and finite by Condition~(1). We show that $h$ is convex. For any $t_1<t_0<t_2$ in $I$, let $\lambda$ be the affine function with $\lambda(t_1)=h(t_1)$ and $\lambda(t_2)=h(t_2)$; then $\lambda$ is positive on $[t_1,t_2]$. Note that the set $D=\Omega\cap(\R\times(t_1,t_2))$ is open, nonempty, and bounded. We apply Lemma~\ref{lem:comparison} to $D$ and the affine function $A(z)+\lambda(y)$ with $z=(x,y)$. Let $z=(x,y)\in\partial D$.
    If $z\notin\Omega$, then $W(z)=0<\lambda(y)$. Otherwise, $y$ equals some $t_i$. Thus $W(x,y)\leq h(y)=\lambda(y)$. Hence, we always have $W(x,y)-\lambda(y)\leq 0$ on $\partial D$, and hence on $D$. Therefore, $W(s,t_0)\leq\lambda(t_0)$ whenever $(s,t_0)\in \Omega$. So $h(t_0)\leq \lambda(t_0)$, which implies that $h$ is convex.

    Suppose $r$ is a finite endpoint of $I$. Then $h(t)\to 0$ as $t\to r$ within $I$. Otherwise, there is a sequence of points in $\Omega$ with $y$-coordinates approaching $r$, at which $W>\sigma$ for some $\sigma>0$. By Condition~(1), they have a limit point $(x_0,e)\in\overline{\Omega}$ with $W(x_0,e)\geq\sigma$. But $e\notin I$ implies that $(x_0,e)\notin\Omega$. Thus $(x_0,e)\in\partial\Omega$, and then $W(x_0,e)=0$, contradiction.

    By Condition~(2), for every $\epsilon>0$, $h(t)\leq\epsilon|t|$ for $t\in I$ with $|t|$ sufficiently large. If $I=(r,\infty)$ with $r$ finite, by Lemma~\ref{lem:convex}, $h$ is non-increasing. Since $h(t)\to 0$ as $t\to r^+$, we then have $h(t)\leq 0$, which contradicts that $h$ is positive. Similarly, $I$ cannot have the form $(-\infty,r)$ for finite $r$. Therefore, $I=\R$. By Lemma~\ref{lem:convex}, $h(t)=h_0$ is constant. 
    % \feng{By Lemma A.5, then $h_0=0$. Something does not seem to be right here.  In the simplifed version of the proof I emailed, it define h(t) to be zero if $t$ is not in $I$ and h is defined on R right away. It proves h is convex on R.} 
    In particular, $W\leq h_0$ on $\overline{\Omega}$.

    Now we let
    \begin{equation}
        M(t)=\{s:(s,t)\in\overline{\Omega}, W(s,t)=h_0\}.
    \end{equation}
    By the previous discussion, on the intersection of $\Omega$ with each horizontal line of height $t$, the maximal of $W(\cdot,t)$ is $h_0$, and it is not realized on $\partial\Omega$ since $W|_{\partial\Omega}=0$. Thus $M(t)$ is nonempty, closed, and bounded. Let $m(t)$ be the largest element in $M(t)$. 
    
    We show that $m(t)$ is convex. Let $t_1<t_0<t_2$, and suppose that $m(t_0)$ exceeds the value at $t_0$ of the chord joining $(m(t_1),t_1)$ and $(m(t_2),t_2)$ by more than some $\eta>0$. Let $\lambda = \lambda(t)$ be the affine function such that $\lambda(t_1)=m(t_1)+\eta$ and $\lambda(t_2)=m(t_2)+\eta$; then we have $\lambda(t_0)<m(t_0)$. We consider the following set
    \begin{equation}
        E=\{(x,y)\in\Omega:t_1<y<t_2,x>\lambda(y)\},
    \end{equation}
    which is bounded and open, and contains $z_0=(m(t_0),t_0)$. Let $C=\max_{\overline{E}}(x-\lambda(y))$.

    Let $F_i=\{s:(s,t_i)\in\overline{\Omega},s\geq \lambda(t_i)=m(t_i)+\eta\}$, for $i=1,2$, which are compact. For $s\in F_i$, we have $s> m(t_i)$, so $W(s,t_i)\neq h_0$, and then $W(s,t_i)< h_0$. When $F_i\neq\varnothing$, there exists $\delta_i>0$ such that for $s\in F_i$, we have $W(s,t_i)\leq h_0-\delta_i$; otherwise, choose any positive $\delta_i$. We then choose $\nu>0$ such that $\nu C<\min\{h_0,\delta_1,\delta_2\}$. Let $A'(z)= A(z)+h_0-\nu(x-\lambda(y))$ be an affine function, and we apply Lemma~\ref{lem:comparison} to $E$ and $A'$, we have
    \begin{equation}
        U-A'=U-A-h_0+\nu(x-\lambda(y))=W-h_0+\nu(x-\lambda(y)).
    \end{equation}
    A point $z=(x,y)\in\partial E$ lies in $\overline{\Omega}$ with $t_1\leq y\leq t_2$ and $x\geq\lambda(y)$. Moreover, at least one of the following holds:
    \begin{itemize}
        \item $z\notin\Omega$: then $z\in\partial\Omega$, and hence $W(z)=0$ and $U(z)-A'(z)\leq -h_0+\nu C<0$;
        \item $x=\lambda(y)$: then $U(z)-A'(z) = W(z) - h_0\leq 0$;
        \item $y=t_i$: then $x\in F_i$, and we have $U(z)-A'(z)\leq -\delta_i+\nu C\leq 0$.
    \end{itemize}
    Therefore we always have $U-A'\leq 0$ on $\partial E$. But $U(z_0)-A'(z_0) = \nu (m(t_0)-\lambda(t_0)) > 0$, which contradicts Lemma~\ref{lem:comparison}. Hence $m(t)$ is a convex function.

    Finally, since $(m(t),t)\in\Omega$, by Condition~(2), for every $\epsilon>0$, $m(t)\leq\epsilon|t|$ when $|t|$ is sufficiently large. By Lemma~\ref{lem:convex}, $m(t)=m_0$ is constant. Thus $m_0\in M(t)$ for all $t$, which implies $(m_0,t)\in\Omega$ for all $t$. Moreover, $U(m_0,t)=pm_0+q+h_0$ is constant with respect to $t$.
\end{proof}

We are now ready to prove Theorem~\ref{thm:AV}.

\begin{proof}[Proof of Theorem~\ref{thm:AV}]
    If $U$ is constant, then we are done. Otherwise by Lemma~\ref{lem:disconnected}, there is an affine function $A$ such that $\{U>A\}$ is disconnected. After composing with a rotation of $\R^2$, we may assume that $A(z)=px+q$ with $p>0$ for $z=(x,y)\in\R^2$.

    Let $\Omega$ be the component of $\{U>A\}$ given by Corollary~\ref{cor:component}. Choosing $c=1$, and we have
    \begin{equation}
        \Omega\cap (\R \times [-T,T])\subset B(0,R(1))\cup([-T,T]\times[-T,T]),
    \end{equation}
    which is bounded, so Condition~(1) of Lemma~\ref{lem:level-line} holds. For Condition~(2), let $0<\epsilon\leq 1$, and set $c=\min\{\epsilon,\epsilon/(4p)\}$. Since $U$ has sublinear growth, there exists $R_1>0$ such that $|U(z)|\leq \epsilon|z|/4$ for $|z|\geq R_1$.
    When $z=(x,y)\in\Omega$ and $|y|\geq \max\{R(c),R_1,4|q|/\epsilon\}$, then $|z|\geq R(c)$. By Corollary~\ref{cor:component} and $z\in\Omega$, $|x|< c|y|\leq \epsilon|y|$. Thus $|z|\leq 2|y|$, and therefore,
    \begin{equation}
        U(z)-A(z)\leq |U(z)| + p|x| + |q| \leq \frac{\epsilon|y|}{2} + \frac{\epsilon|y|}{4} + \frac{\epsilon|y|}{4} = \epsilon|y|.
    \end{equation}
    By Lemma~\ref{lem:level-line}, $\Omega$ contains a vertical level line.

    Let
    \begin{equation}
        S=\{t\in\R: U(t,y)=U(t,0), \forall y\in\R\}
    \end{equation}
    denote the locations of all vertical level lines of $U$.
    Then $S$ is nonempty, and closed since $U$ is continuous. We show that $S=\R$, which then proves the theorem.
    
    Suppose otherwise that $S\neq\R$. Then the nonempty open set $\R\backslash S$ contains a component $(a,b)$. Note that at least one endpoint of $(a,b)$ is finite, and the finite endpoints of $(a,b)$ is contained in $S$. By applying a rotation of $\pi$ of the plane, we may assume that $a$ is finite. Let $\alpha$ be the value of $U$ on $\{x=a\}$ and $H=(a,b)\times\R$. 
    
    We will find an affine function $A'(z)=p'x+q'$ and a component $\Omega'\subset H$ of ${U>A'}$, for $U$ or $-U$, that satisfies Conditions~(1) and (2) of Lemma~\ref{lem:level-line}. Then it provides a vertical level line of $U$ in $H$, which contradicts $S\cap (a,b)=\varnothing$.

    \textbf{Case 1: $b<\infty$}. Let $\beta$ be the value of $U$ on $\{x=b\}$. Let $A'(z)=p'x+q'$ be the affine function such that $A'(a,0)=\alpha$ and $A'(b,0)=\beta$. If $U=A'$ on $H$, then every vertical line in $H$ is a level line, which contradicts $S\cap (a,b)=\varnothing$. So $U(z_0)\neq A'(z_0)$ for some $z_0\in H$. By replacing $U$ with $-U$ if necessary, we may assume $U(z_0)>A'(z_0)$. Let $\Omega'$ be the component of $\{U>A'\}$ that contains $z_0\in H$. Since $U=A'$ on $\partial H$, we know that $\Omega'\subset H$.

    Since $H$ is a strip, Condition~(1) of Lemma~\ref{lem:level-line} holds. For Condition~(2), note that for any $z=(x,y)\in \Omega'$, $|x|\leq c_0=\max\{|a|,|b|\}$ and $|A'(z)|\leq C_0=\max\{|\alpha|,|\beta|\}$. Let $0<\epsilon<1$ and $R_1$ be as above; then for $z\in\Omega'$ with $|y|\geq\max\{R_1,c_0/\epsilon,2C_0/\epsilon\}$, we have $|x|\leq\epsilon|y|$, $|z|\leq 2|y|$, and 
    \begin{equation}
        U(z)-A'(z)\leq \frac{\epsilon|y|}{2} + C_0 \leq \epsilon|y|.
    \end{equation}
    Hence Condition~(2) also holds.

    \textbf{Case 2: $b=\infty$}. Now $H=\{x>a\}$. If $U(z)=\alpha$ on $H$, then every vertical line in $H$ is a level line, which is impossible; thus $U(z_0)\neq \alpha$ for some $z_0=(x_0,y_0)\in H$. By the same reason, we may assume $U(z_0)>\alpha$. Choose $p'>0$ such that $p'(x_0-a) + \alpha < U(z_0)$, and we define $A'(z)=p'(x-a)+\alpha$. Then $A'(z_0)<U(z_0)$. Let $\Omega'$ be the component of $\{U>A'\}$ that contains $z_0$. Since $U=A'$ on $\partial H$ and $z_0\in H$, we have $\Omega'\subset H$.

    Choose $c=1$ and apply Lemma~\ref{lem:sublinear-property} for $A'$. Then there exists $R'(1)>0$ such that $\Omega'\subset\{a<x<|y|\}\cup B(0,R'(1))$, so Condition~(1) holds. For Condition~(2), let $0<\epsilon\leq 1$ and $c=\min\{\epsilon,\epsilon/(4p')\}$. Then for every $z\in\Omega'$ with
    \begin{equation}
        |y|\geq\max\left\{R'(c), R_1, \frac{|a|}{\epsilon},\frac{4(|\alpha|+2p'|a|)}{\epsilon}\right\}, 
    \end{equation}
    we have $|x|\leq\max\{|a|,c|y|\}\leq\epsilon|y|$, and $|x|\leq c|y|+|a|$. Thus,
    \begin{equation}
        U(z)-A'(z)\leq |U(z)| + p'|x| + p'|a| + |\alpha|\leq \frac{\epsilon|y|}{2} + \frac{\epsilon|y|}{4} + 2p'|a| + |\alpha| \leq\epsilon|y|. 
    \end{equation}

    In both cases, Lemma~\ref{lem:level-line} yields a vertical level line of $U$ in $\Omega'\subset H$, which is a contradiction. Hence, $S=\R$.
\end{proof}

\bibliographystyle{alpha}
\bibliography{bibliography-mathscinet}
\end{document}

%% file: fig_dual_weights.tikz
% Three panels: reciprocal duality, cotangent weights, and Dubejko weights.
% Requires \usepackage{amsmath,tikz} and \usetikzlibrary{angles,quotes}.
% L and R refer to the left and right of the directed edge q -> p.
% J(x,y)=(-y,x); all three illustrated weights are positive.
% Dubejko formula: https://arxiv.org/pdf/1910.12796, Definition 1.2.
\begin{tikzpicture}[
  color=black,font=\footnotesize,line cap=round,line join=round,
  primal/.style={line width=0.65pt},
  dual/.style={line width=1pt},
  circle line/.style={draw=black,line width=0.5pt},
  point/.style={circle,fill=black,inner sep=1.15pt},
  center point/.style={circle,fill=black,inner sep=1.45pt},
  panel title/.style={font=\small},
  formula/.style={font=\small,align=center,inner sep=0pt}
]
  % Fixed panel positions; coordinates are in centimetres.
  \path[use as bounding box] (0,-1.15) rectangle (16.15,4.8);
  \node[panel title] at (2.55,4.6) {(a) Dual equation};
  \node[panel title] at (8.05,4.6) {(b) Cotangent weight};
  \node[panel title] at (13.55,4.6) {(c) Dubejko weight};

  % (a) Local boundaries of two possibly unbounded polygonal faces sharing pq.
  \begin{scope}[shift={(2.55,2.1)},scale=0.88]
    \coordinate (AP) at (-0.35,0.75);
    \coordinate (AQ) at (0.35,-0.75);
    % Reciprocal-dual vertices, not polygon centers. These positions satisfy
    % g_L-g_R=(53/30)J(p-q); containment in the faces is only illustrative.
    \coordinate (AL) at (-1.4,{-1.4*7/15});
    \coordinate (AG) at (1.25,{1.25*7/15});
    \coordinate (AM) at (0,0);
    % Dashed ends indicate omitted boundary portions, not closing edges.
    % This local view imposes no boundedness assumption on either face.
    \draw[primal] (AP)--(-1.35,0.75)--(-2.05,0.15)--(-2.15,{-5/12});
    \draw[primal,densely dashed] (-2.15,{-5/12})--(-2.2,-0.7);
    \draw[primal] (AP)--(0.6,2.05)--(1.7,{2.05+1.1*0.45/1.45});
    \draw[primal,densely dashed]
      (1.7,{2.05+1.1*0.45/1.45})--(2.05,2.5);
    \draw[primal] (AQ)--(0.2,-1.55)--(-0.15,{-1.55-0.35*2/3});
    \draw[primal,densely dashed]
      (-0.15,{-1.55-0.35*2/3})--(-0.4,-1.95);
    \draw[primal] (AQ)--(1.3,-1.05)--(2.1,-0.75)--(2.34,-0.3);
    \draw[primal,densely dashed] (2.34,-0.3)--(2.5,0);
    \draw[primal] (AQ)--(AP);
    \draw[dual] (AG)--(AL);
    % Orient the shared edge from q to p, as in the dual equation.
    \draw[primal,->,>=stealth] (AQ)--(-0.175,0.375);
    \pic[draw,line width=0.4pt,angle radius=2mm]
      {right angle=AQ--AM--AG};
    \node[point,label={above left:{$p$}}] at (AP) {};
    \node[point,label={[label distance=2pt]below right:{$q$}}] at (AQ) {};
    \node[center point,label={below:{$g_L$}}] at (AL) {};
    \node[center point,label={below right:{$g_R$}}] at (AG) {};
    \node at (-1.2,0.1) {$L$};
    \node at (1.2,1.5) {$R$};
  \end{scope}
  \node[formula] at (2.55,-0.25)
    {$g_L-g_R=c_{pq}J(p-q)$};
  \node[formula] at (2.55,-0.88)
    {$\displaystyle c_{pq}=\frac{|g_L-g_R|}{|p-q|}$};

  % (b) The marked angles at the circumcenters equal the opposite angles
  % of the incident triangles in this illustrated configuration.
  \begin{scope}[shift={(7.68,1.98)},scale=0.67]
    \pgfmathsetmacro{\leftRadius}{sqrt(5/2)}
    \pgfmathsetmacro{\rightRadius}{sqrt(89/20)}
    \coordinate (BP) at (-0.5,1);
    \coordinate (BQ) at (0.5,-1);
    \coordinate (BL) at (-1,-0.5);
    \coordinate (BG) at (1.6,0.8);
    \coordinate (BM) at (0,0);
    \coordinate (BS) at ({-1-3*\leftRadius/5},
                            {-0.5-4*\leftRadius/5});
    \coordinate (BR) at ({1.6+40*\rightRadius/41},
                            {0.8+9*\rightRadius/41});
    \draw[circle line] (BL) circle[radius=\leftRadius];
    \draw[circle line] (BG) circle[radius=\rightRadius];
    % Mask only the circumcircle behind this label; redraw the straight
    % edges afterwards so that their geometry stays fully visible.
    \node[fill=white,font=\fontsize{7}{8}\selectfont,inner sep=0.4pt]
      at (-0.46,0.13) {\phantom{$\beta_{pq}$}};
    \draw[primal] (BP)--(BS)--(BQ)--(BR)--cycle;
    \draw[primal] (BP)--(BQ);
    \draw[dual] (BL)--(BG);
    \draw[draw=black,line width=0.55pt] (BP)--(BG);
    \draw[draw=black,line width=0.55pt] (BP)--(BL);
    \pic[draw,line width=0.4pt,angle radius=1.8mm]
      {right angle=BP--BM--BG};
    \pic[draw=black,text=black,
         line width=0.4pt,angle radius=2.3mm]
      {angle=BG--BL--BP};
    \node[text=black,font=\fontsize{7}{8}\selectfont,inner sep=0.4pt]
      at (-0.46,0.13) {$\beta_{pq}$};
    \pic[draw=black,text=black,
         line width=0.4pt,angle radius=4mm,
         angle eccentricity=1.65,"$\alpha_{pq}$"]
      {angle=BP--BG--BL};
    % Repeat the equal angles at the opposite vertices on the circumcircles.
    \pic[draw=black,text=black,line width=0.4pt,angle radius=4.2mm,
         angle eccentricity=1.6,"$\beta_{pq}$"]
      {angle=BQ--BS--BP};
    \pic[draw=black,text=black,line width=0.4pt,angle radius=4.4mm,
         angle eccentricity=1.6,"$\alpha_{pq}$"]
      {angle=BP--BR--BQ};
    \node[point,label={above left:{$p$}}] at (BP) {};
    \node[point,label={below right:{$q$}}] at (BQ) {};
    \node[point] at (BS) {};
    \node[point] at (BR) {};
    \node[center point,label={below right:{$g_L$}}] at (BL) {};
    \node[center point,label={[label distance=2pt]below:{$g_R$}}] at (BG) {};
  \end{scope}
  \node[formula] at (8.05,-0.25) {Circumcenters $g_L,g_R$};
  \node[formula] at (8.05,-0.88)
    {$\displaystyle w_{pq}=\frac12
      \bigl(\cot\alpha_{pq}+\cot\beta_{pq}\bigr)$};

  % (c) Four tangent vertex circles and two tangent face incircles.
  % Exact side lengths are sums of the endpoint circle radii.
  \begin{scope}[shift={(13.55,3.1)},scale=0.6]
    \pgfmathsetmacro{\radiusP}{1.15}
    \pgfmathsetmacro{\radiusQ}{0.95}
    \pgfmathsetmacro{\radiusS}{1.9}
    \pgfmathsetmacro{\radiusR}{1.45}
    \pgfmathsetmacro{\edgeLength}{\radiusP+\radiusQ}
    \pgfmathsetmacro{\lowerX}{((\radiusP+\radiusS)^2
      +\edgeLength^2-(\radiusQ+\radiusS)^2)/(2*\edgeLength)}
    \pgfmathsetmacro{\lowerY}{sqrt((\radiusP+\radiusS)^2-\lowerX^2)}
    \pgfmathsetmacro{\upperX}{((\radiusP+\radiusR)^2
      +\edgeLength^2-(\radiusQ+\radiusR)^2)/(2*\edgeLength)}
    \pgfmathsetmacro{\upperY}{sqrt((\radiusP+\radiusR)^2-\upperX^2)}
    \pgfmathsetmacro{\inradiusL}{sqrt(\radiusP*\radiusQ*\radiusS
      /(\radiusP+\radiusQ+\radiusS))}
    \pgfmathsetmacro{\inradiusR}{sqrt(\radiusP*\radiusQ*\radiusR
      /(\radiusP+\radiusQ+\radiusR))}
    \begin{scope}[rotate=-50]
      \coordinate (CP) at (0,0);
      \coordinate (CQ) at (\edgeLength,0);
      \coordinate (CS) at (\lowerX,-\lowerY);
      \coordinate (CR) at (\upperX,\upperY);
      \coordinate (CL) at (\radiusP,-\inradiusL);
      \coordinate (CG) at (\radiusP,\inradiusR);
      \coordinate (CT) at (\radiusP,0);
    \end{scope}
    \draw[circle line] (CP) circle[radius=\radiusP];
    \draw[circle line] (CQ) circle[radius=\radiusQ];
    \draw[circle line] (CS) circle[radius=\radiusS];
    \draw[circle line] (CR) circle[radius=\radiusR];
    \draw[line width=0.55pt] (CL) circle[radius=\inradiusL];
    \draw[line width=0.55pt] (CG) circle[radius=\inradiusR];
    \draw[primal] (CP)--(CS)--(CQ)--(CR)--cycle;
    \draw[primal] (CP)--(CQ);
    \draw[dual] (CL)--(CG);
    \pic[draw,line width=0.4pt,angle radius=1.4mm]
      {right angle=CP--CT--CG};
    \node[point,label={above left:{$p$}}] at (CP) {};
    \node[point,label={below right:{$q$}}] at (CQ) {};
    \node[point] at (CS) {};
    \node[point] at (CR) {};
    \node[center point] at (CL) {};
    \node[center point] at (CG) {};
    \node[fill=white,inner sep=0.7pt,below=2pt] at (CL) {$g_L$};
    \node[fill=white,inner sep=0.7pt,above=2pt] at (CG) {$g_R$};
  \end{scope}
  \node[formula] at (13.55,-0.25) {Incenters $g_L,g_R$};
  \node[formula] at (13.55,-0.88)
    {$\displaystyle c_{pq}=\frac{\rho_L+\rho_R}{r_p+r_q}$};
\end{tikzpicture}

%% file: fig_transient.tex
% Figure for Section 5: the transient Delaunay triangulation (nested rectangles).
% Drawn to scale from the exact coordinates; needs only TikZ and the angles
% library, which the preamble already loads.
\begin{figure}[h!]
\centering
\tikzset{
  dt edge/.style={gray!75!black, line width=0.5pt},
  dt diag/.style={gray!55, dashed, line width=0.45pt},
  dt ray/.style={blue!65!black, line width=1.3pt},
  dt vertex/.style={circle, fill=black, inner sep=0pt, minimum size=3pt},
}
\begin{tikzpicture}[scale=0.62, font=\small, baseline=(panel.base)]
  % ring n has corners (+-4^n, +-2^n)
  \foreach \n/\x/\y in {0/1/1, 1/4/2, 2/16/4} {
    \coordinate (A\n) at (\x,\y);   \coordinate (B\n) at (\x,-\y);
    \coordinate (C\n) at (-\x,-\y); \coordinate (D\n) at (-\x,\y);
  }
  \coordinate (O) at (0,0);
  % the two triangles containing the edge A_0A_1
  \fill[blue!6] (A0) -- (A1) -- (B1) -- cycle;
  \fill[blue!6] (A0) -- (A1) -- (D1) -- cycle;
  \begin{scope}
    \clip (-5.6,-2.75) rectangle (6.1,2.75);
    % circumcircle of the cyclic trapezoid A_0B_0B_1A_1: center (3,0), radius sqrt(5)
    \draw[gray!60, densely dotted, line width=0.6pt] (3,0) circle ({sqrt(5)});
    \foreach \p/\q in {A0/B1, A0/D1, C0/D1, C0/B1} \draw[dt diag] (\p) -- (\q);
    \foreach \v in {A0,B0,C0,D0} \draw[dt edge] (O) -- (\v);
    \foreach \n in {0,1} \draw[dt edge] (A\n) -- (B\n) -- (C\n) -- (D\n) -- cycle;
    \foreach \v in {A,B,C,D} \draw[dt edge] (\v0) -- (\v1) -- (\v2);
    \draw[dt ray] (O) -- (A0) -- (A1);
  \end{scope}
  % the ray continues toward A_2 = (16,4); it leaves the frame at x = 6.1
  \draw[dt ray, ->] (A1) -- (6.1,2.35);
  % the angles opposite to A_0A_1
  \draw pic[draw, blue!65!black, angle radius=4.2mm] {angle=A1--B1--A0};
  \draw pic[draw, blue!65!black, angle radius=16mm] {angle=A0--D1--A1};
  \node[font=\scriptsize, text=blue!65!black] at (3.62,-1.05) {$\beta_0$};
  \node[font=\scriptsize, text=blue!65!black] at (-0.62,1.66) {$\delta_0$};
  \foreach \v in {O,A0,B0,C0,D0,A1,B1,C1,D1} \node[dt vertex] at (\v) {};
  \node[right=2pt] at (O) {$O$};
  \node[above] at (A0) {$A_0$};  \node[above right=-1pt] at (B0) {$B_0$};
  \node[below] at (C0) {$C_0$};  \node[left] at (D0) {$D_0$};
  \node[above] at (A1) {$A_1$};  \node[below] at (B1) {$B_1$};
  \node[below] at (C1) {$C_1$};  \node[above] at (D1) {$D_1$};
  \node (panel) at (0.25,-3.4) {(a)};
\end{tikzpicture}\hfill
\begin{tikzpicture}[scale=0.19, font=\small, baseline=(panel.base)]
  \foreach \n/\x/\y in {0/1/1, 1/4/2, 2/16/4, 3/64/8} {
    \coordinate (A\n) at (\x,\y);   \coordinate (B\n) at (\x,-\y);
    \coordinate (C\n) at (-\x,-\y); \coordinate (D\n) at (-\x,\y);
  }
  \coordinate (O) at (0,0);
  \begin{scope}
    \clip (-19.5,-5.2) rectangle (19.5,5.2);
    \foreach \n/\m in {0/1, 1/2} \foreach \p/\q in {A/B, A/D, C/D, C/B}
      \draw[dt diag] (\p\n) -- (\q\m);
    \foreach \v in {A0,B0,C0,D0} \draw[dt edge] (O) -- (\v);
    \foreach \n in {0,1,2} \draw[dt edge] (A\n) -- (B\n) -- (C\n) -- (D\n) -- cycle;
    \foreach \v in {A,B,C,D} \draw[dt edge] (\v0) -- (\v1) -- (\v2) -- (\v3);
    \draw[dt ray] (O) -- (A0) -- (A1) -- (A2);
  \end{scope}
  % the ray continues toward A_3 = (64,8); it leaves the frame at x = 19.5
  \draw[dt ray, ->] (A2) -- (19.5,4.29);
  \foreach \v in {A1,B1,C1,D1,A2,B2,C2,D2} \node[dt vertex, minimum size=2.4pt] at (\v) {};
  \node[above left=-1pt] at (A1) {$A_1$};
  \node[above] at (A2) {$A_2$};  \node[below] at (B2) {$B_2$};
  \node[below] at (C2) {$C_2$};  \node[above] at (D2) {$D_2$};
  \node (panel) at (0,-8.9) {(b)};
\end{tikzpicture}
\caption{The Delaunay triangulation $T$ of $V$; the $n$-th rectangle has corners $(\pm4^n,\pm2^n)$.  The thick path $OA_0A_1A_2\cdots$ is the ray used in the resistance estimate. (a)~Near the origin. The dotted circle through $A_0,B_0,B_1,A_1$ contains no other vertex, and the edge $A_0A_1$ has weight $\frac12(\cot\beta_0+\cot\delta_0)=\frac12(1+5)=3$, where $\beta_0=\angle A_0B_1A_1$ and $\delta_0=\angle A_0D_1A_1$. (b)~The first three rectangles.}
\label{fig:transient}
\end{figure}

%% file: mix1.tikz
% Two labeled diagrams: a mixed weighted triangle and its adjacent-face dual.
% Preamble: \usepackage{amsmath,tikz}
%           \usetikzlibrary{angles,arrows.meta}
% Usage: \input{figures/mix1.tikz}
% Geometry: p_i=(0,0), p_j=(5,0), p_k=(1.75,4),
% r_i=1.3, r_k=1.375, and (epsilon_i,epsilon_j,epsilon_k)=(1,0,-1).
% Signed site weights: (s_i,s_j,s_k)=(1.69,0,-1.890625).
% q_ijk=(2.669,1.662703125) equalizes |q-p_a|^2-s_a.
% Each c_ab is the orthogonal projection of q_ijk onto the edge p_a p_b.
% The dashed circle is only a guide to the virtual-radius magnitude.
% Panel (b) adds p_ell=(3.3,-3.7), s_ell=0.9^2, giving
% q_ijell=(2.669,-1.060081081081). Circles are omitted in panel (b)
% to make the shared edge, perpendicular dual edge, and heights clear.
% These definitions are local, so the figure does not alter document colors.
\begingroup
\definecolor{mixInk}{RGB}{47,59,71}
\definecolor{mixMuted}{RGB}{132,141,149}
\definecolor{mixReal}{RGB}{76,105,175}
\definecolor{mixVirtual}{RGB}{128,91,167}
\definecolor{mixDual}{RGB}{65,124,132}
\definecolor{mixCenter}{RGB}{187,108,54}
\definecolor{mixFace}{RGB}{243,247,247}
\begin{tikzpicture}[
  font=\footnotesize,text=mixInk,line cap=round,line join=round,
  primal/.style={draw=mixInk,line width=0.8pt},
  dual/.style={draw=mixDual,line width=1.05pt},
  vertex/.style={circle,fill=mixInk,draw=white,line width=0.3pt,
                 inner sep=1.55pt},
  edge center/.style={rectangle,draw=mixDual,fill=white,
                     line width=0.7pt,inner sep=1.3pt},
  right angle mark/.style={draw=mixMuted,line width=0.45pt,
                           angle radius=1.8mm},
  dimension/.style={draw=mixMuted,line width=0.6pt,
                   {Latex[length=1.5mm,width=1mm]}-{Latex[length=1.5mm,width=1mm]}},
  point label/.style={inner sep=1pt},
  panel title/.style={font=\small,text=mixInk,align=center,text width=7.5cm}
]
 % \path[use as bounding box] (0,-0.1) rectangle (16,7.2);
  % \node[panel title] at (4,6.85)
  %   {(a) Dual center of a weighted triangle};
  % \node[panel title] at (12,6.85)
  %   {(b) Dual structure across two adjacent triangles};
  % Parameters of the illustrated configuration.
  \pgfmathsetmacro{\mixLength}{5}
  \pgfmathsetmacro{\mixKx}{1.75}
  \pgfmathsetmacro{\mixKy}{4}
  \pgfmathsetmacro{\mixRi}{1.3}
  \pgfmathsetmacro{\mixRk}{1.375}
  \pgfmathsetmacro{\mixSi}{(\mixRi)^2}
  \pgfmathsetmacro{\mixSj}{0}
  \pgfmathsetmacro{\mixSk}{-(\mixRk)^2}
  \pgfmathsetmacro{\mixQx}{((\mixLength)^2+\mixSi-\mixSj)/(2*\mixLength)}
  \pgfmathsetmacro{\mixQy}{((\mixKx)^2+(\mixKy)^2+\mixSi-\mixSk
                            -2*\mixQx*\mixKx)/(2*\mixKy)}
  \pgfmathsetmacro{\mixTik}{(\mixQx*\mixKx+\mixQy*\mixKy)
                             /((\mixKx)^2+(\mixKy)^2)}
  \pgfmathsetmacro{\mixTjk}{((\mixQx-\mixLength)*(\mixKx-\mixLength)
                              +\mixQy*\mixKy)
                             /((\mixKx-\mixLength)^2+(\mixKy)^2)}
  \begin{scope}[shift={(1.55,1.7)},scale=0.79]
  \coordinate (PI) at (0,0);
  \coordinate (PJ) at (\mixLength,0);
  \coordinate (PK) at (\mixKx,\mixKy);
  \coordinate (Q) at (\mixQx,\mixQy);
  \coordinate (CIJ) at (\mixQx,0);
  \coordinate (CIK) at ({\mixTik*\mixKx},{\mixTik*\mixKy});
  \coordinate (CJK) at ({\mixLength+\mixTjk*(\mixKx-\mixLength)},
                        {\mixTjk*\mixKy});

  \fill[mixFace] (PI)--(PJ)--(PK)--cycle;
  \draw[primal] (PI)--(PJ)--(PK)--cycle;
  \draw[draw=mixReal,line width=0.85pt] (PI) circle[radius=\mixRi];
  \draw[draw=mixVirtual,line width=0.85pt,dash pattern=on 3pt off 2pt]
    (PK) circle[radius=\mixRk];
  \draw[dual] (CIK)--(Q)--(CJK);
  \draw[dual] (Q)--(CIJ);

  \pic[right angle mark] {right angle=PJ--CIJ--Q};
  \pic[right angle mark] {right angle=PK--CIK--Q};
  \pic[right angle mark] {right angle=PK--CJK--Q};
  \node[edge center] at (CIJ) {};
  \node[edge center] at (CIK) {};
  \node[edge center] at (CJK) {};
  \node[vertex] at (PI) {};
  \node[vertex] at (PJ) {};
  \node[vertex] at (PK) {};
  \node[circle,fill=mixCenter,draw=white,line width=0.5pt,inner sep=2.1pt]
    at (Q) {};

  \node[point label,text=mixReal] at (0,{-\mixRi-0.36})
    {$p_i,\ \varepsilon_i=+1$};
  \node[point label,below right=6pt] at (PJ)
    {$p_j,\ \varepsilon_j=0$};
  \node[point label,text=mixVirtual] at (\mixKx,{\mixKy+\mixRk+0.34})
    {$p_k,\ \varepsilon_k=-1$};
  \node[point label,text=mixReal] at (-1.15,1.15) {$r_i$};
  \node[point label,text=mixVirtual]
    at ({\mixKx+1.40},{\mixKy+0.83}) {$r_k$};
  \node[point label,text=mixDual,above right=3pt] at (CIJ) {$c_{ij}$};
  \node[point label,text=mixDual,above left=7pt] at (CIK) {$c_{ik}$};
  \node[point label,text=mixDual,above right=7pt] at (CJK) {$c_{jk}$};
  \node[point label,text=mixCenter,below right=5pt] at (Q) {$z_{ijk}$};

  % \draw[dimension] (0,-0.32)--(\mixQx,-0.32)
  %   node[midway,below=4pt,text=mixInk] {$d_{ij}$};
  % \draw[dimension] (\mixQx,-0.32)--(\mixLength,-0.32)
  %   node[midway,below=4pt,text=mixInk] {$d_{ji}$};
  \end{scope}

  % (b) Both face centers equalize powers at i and j, so they share
  % the x-coordinate mixQx. Their connecting segment is perpendicular
  % to p_i p_j; both displayed signed heights are positive here.
  \begin{scope}[shift={(9.8,3.25)},scale=0.72]
    \pgfmathsetmacro{\mixLx}{3.3}
    \pgfmathsetmacro{\mixLy}{-3.7}
    \pgfmathsetmacro{\mixSl}{0.9^2}
    \pgfmathsetmacro{\mixLowerY}{((\mixLx)^2+(\mixLy)^2+\mixSi-\mixSl
      -2*\mixQx*\mixLx)/(2*\mixLy)}
    \coordinate (BI) at (0,0);
    \coordinate (BJ) at (\mixLength,0);
    \coordinate (BK) at (\mixKx,\mixKy);
    \coordinate (BL) at (\mixLx,\mixLy);
    \coordinate (BQK) at (\mixQx,\mixQy);
    \coordinate (BQL) at (\mixQx,\mixLowerY);
    \coordinate (BC) at (\mixQx,0);
    \fill[mixFace] (BI)--(BK)--(BJ)--(BL)--cycle;
    \draw[primal] (BI)--(BK)--(BJ)--(BL)--cycle;
    \draw[primal,line width=0.95pt] (BI)--(BJ);
    \draw[dual,line width=1.2pt] (BQK)--(BQL);
    \pic[right angle mark] {right angle=BJ--BC--BQK};
    \foreach \v in {BI,BJ,BK,BL}
      \node[vertex] at (\v) {};
    \foreach \v in {BQK,BQL}
      \node[circle,fill=mixCenter,draw=white,line width=0.5pt,inner sep=2.1pt]
        at (\v) {};
    \node[edge center] at (BC) {};
    \node[point label,left=4pt] at (BI) {$p_i$};
    \node[point label,right=4pt] at (BJ) {$p_j$};
    \node[point label,above=4pt] at (BK) {$p_m$};
    \node[point label,below=4pt] at (BL) {$p_k$};
    \node[point label,text=mixCenter,above left=4pt] at (BQK) {$z_{ijm}$};
    \node[point label,text=mixCenter,below right=4pt] at (BQL) {$z_{ijk}$};
    \node[point label,text=mixDual,below right=4pt] at (BC) {$c_{ij}$};
    % \node[point label,text=mixDual,anchor=east]
    %   at ({\mixQx-0.18},{\mixQy/2}) {$h_{ij}^{k}$};
    % \node[point label,text=mixDual,anchor=east]
    %   at ({\mixQx-0.18},{\mixLowerY/2}) {$h_{ij}^{\ell}$};
  \end{scope}
  \node[text=mixDual] at (11.8,-0.1)
    {$z_{ijk}-z_{ijm}\perp p_i-p_j$};
\end{tikzpicture}
\endgroup